\documentclass[10pt]{amsart}
\usepackage{amsmath}
\usepackage{amssymb}
\usepackage{epsbox}
\usepackage{amsthm}
\usepackage{graphicx}

 \newtheorem{thm}{Theorem}
 \newtheorem{cor}{Corollary}
 \newtheorem{prop}{Proposition}
 \newtheorem{lem}{Lemma}
  \newtheorem{claim}{Claim}
 
 {\theoremstyle{definition}
 \newtheorem{rem}{Remark}}
 {\theoremstyle{definition}
  \newtheorem{defn}{Definition}}
 {\theoremstyle{definition}
 \newtheorem{exam}{Example}}
 {\theoremstyle{definition}
 }
  {\theoremstyle{definition}
 \newtheorem{ques}{Question}}

 \newcommand{\LO}{\textrm{LO}}
 \newcommand{\Z}{\mathbb{Z}}
  \newcommand{\R}{\mathbb{R}}

  \newcommand{\tE}{\widetilde{E} }
 \newcommand{\mD}{\mathcal{D} }
  \newcommand{\mA}{\mathcal{A} }
   \newcommand{\mX}{\mathcal{X} }
 \newcommand{\mS}{\mathcal{S} }

\title{Dehornoy-like left orderings and isolated left orderings}
\author{Tetsuya Ito}
\address{Graduate School of Mathematical Science, University of Tokyo, 3-8-1 Komaba Meguro-ku Tokyo 153-8914, Japan}
\email{tetitoh@ms.u-tokyo.ac.jp}
\urladdr{http://ms.u-tokyo.ac.jp/~tetitoh}

\begin{document}

\begin{abstract} 
We introduce a Dehornoy-like ordering of groups, which is a generalization of the Dehornoy ordering of the braid groups. Under a weak assumption which we call Property $F$, we show that Dehornoy-like orderings have properties similar to the Dehornoy ordering, and produce isolated left orderings. We also construct new examples of Dehornoy-like ordering and isolated orderings and study their more precise properties.  
\end{abstract}
 \maketitle

\section{Introduction}

A {\it left-ordering} of a group $G$ is a total ordering $<_{G}$ of $G$ preserved by the left action of $G$ itself. That is, $g <_{G} g'$ implies $hg<_{G}hg'$ for all $g,g',h \in G$.
A group $G$ is {\it left-orderable} if $G$ has at least one left-ordering.

One of the most important left ordering is the {\it Dehornoy ordering} of the braid group $B_{n}$. The Dehornoy ordering has a simple, but still mysterious definition which uses a special kind of word representatives called $\sigma$-positive words. The Dehornoy ordering can be regarded as the most natural left ordering of the braid groups, but its combinatorial structure is rather complicated. 

In this paper we introduce a {\it Dehornoy-like ordering} of groups. This is a left-ordering defined in a similar way to the Dehornoy ordering. The aim of this paper is to study Dehornoy-like orderings and give new examples of Dehornoy-like orderings.

The study of the Dehornoy-like ordering produces another interesting family of left-orderings.
Recall that the {\it positive cone} $P=\{g \in G\:|\: 1 <_{G}g\}$ of a left ordering $<_{G}$ has the following two properties {\bf LO1} and {\bf LO2}.
\begin{description}
\item[LO1] $P \cdot P \subset P$.
\item[LO2] $G = P \coprod \{1\} \coprod P^{-1}$.
\end{description}

Conversely, for a subset $P$ of $G$ having the properties {\bf LO1} and {\bf LO2} one can obtain a left-ordering $<_{G}$ by defining $h <_{G} g$ if $h^{-1}g \in P$. 
Thus, the set of all left-orderings of $G$, which we denote by $\LO(G)$, is naturally regarded as a subset of the powerset $2^{G-\{1\}} = \{+,-\}^{G-\{1\}}$. We equip a discrete topology on $2=\{+,-\}$ and equip the power set topology on $2^{G-\{1\}}$. This induces a topology on $\LO(G)$ as the subspace topology. $\LO(G)$ is compact, totally disconnected, and metrizable \cite{s}. It is known that for a countable group $G$, $\LO(G)$ is either finite or uncountable \cite{l}. So $\LO(G)$ is very similar to the Cantor set if $G$ has infinitely many left orderings.

An {\it isolated ordering} is a left-ordering which corresponds to an isolated point of $\LO(G)$.
Isolated orderings are easily characterized by their positive cones. Observe that by {\bf LO1}, the positive cone of a left ordering is a submonoid of $G$. A left-ordering $<$ is isolated if and only if its positive cone is finitely generated as a submonoid of $G$. 

We begin with a systematic study of Dehornoy-like orderings in Section 2.
We introduce a property called {\it Property $F$} for an ordered finite generating set $\mS$, which plays an important role in our study of Dehornoy-like orderings.
We show that Property $F$ allows us to relate Dehornoy-like orderings and isolated orderings in a very simple way. Moreover, using Property $F$ we generalize known properties of the Dehornoy ordering of $B_{n}$ to Dehornoy-like orderings of a general group $G$. 

In section 3, we construct a new example of Dehornoy-like and isolated orderings and study their detailed properties. Our examples are generalization of Navas' example of Dehornoy-like and isolated orderings given in \cite{n2}.

We consider the groups of the form $\Z *_{\Z} \Z$, the amalgamated free product of two infinite cyclic groups. Such a group is presented as
\[ G_{m,n} = \langle x,y \: | \: x^{m}=y^{n} \rangle\;\;\; (m \geq n).\] 
The groups $G_{m,n}$ appear in many contexts.
Observe that $G_{2,2}$ is the Klein bottle group and $G_{3,2}$ is the 3-strand braid group $B_{3}$. For coprime $(m,n)$, $G_{m,n}$ is nothing but the fundamental group of the complement of the $(m,n)$-torus knot. The family of groups $\{G_{m,2}\}$ are the central extension of the Hecke groups, studied by Navas in \cite{n2}. 
We always assume $(m,n) \neq (2,2)$ because the Klein bottle group $G_{2,2}$ is exceptional since it admits only finitely many left orderings. As we will see later, other groups $G_{m,n}$ have infinitely many (hence uncountably many) left orderings.

To give a Dehornoy-like and an isolated ordering, we introduce generating sets $\mS=\{s_{1}=xyx^{-m+1}, s_{2}=x^{m-1}y^{-1}\}$ and $\mA =\{a = x,b = yx^{-m+1}\}$. Using the generator $\mS$, the group $G_{m,n}$ is presented as
\[ G_{m,n}  =  \langle s_{1},s_{2} \: | \: s_{2}s_{1}s_{2}=
( (s_{1}s_{2})^{m-2}s_{1})^{n-1} \rangle \]
Observe that for $(m,n)=(3,2)$, this presentation agrees with the standard presentation of the 3-strand braid group $B_{3}$. Similarly, using the generators $\mA$, the group $G_{m,n}$ is presented as
\[ G_{m,n} = \langle a,b \: | \: (ba^{m-1})^{n-1}b=a \rangle.\]
For $n=2$, the above presentation coincide with Navas' presentation of $G_{m,2}$.
We will show that $\mS$ defines a Dehornoy-like ordering $<_{D}$ of $G_{m,n}$ and $\mA$ defines an isolated ordering $<_{A}$ of $G_{m,n}$ in Theorem \ref{thm:order}.

We will also give an alternative description of the Dehornoy-like ordering $<_{D}$ of $G_{m,n}$ in Theorem \ref{thm:dynamics} by using the action on the Bass-Serre tree.
Such an action is natural since $G_{m,n}=\Z*_{\Z} \Z$ is an amalgamated free product.
In this point of view, the Dehornoy-like ordering $<_{D}$ can be regarded as an natural left ordering of $G_{m,n}$ like the Dehonroy ordering of $B_{3}$, although the combinatorial definition seems to be quite strange.

The dynamics of ordering allows us to give more detailed properties of the Dehornoy-like ordering $<_{D}$. In Theorem \ref{thm:propertyS}, we will show that one particular property of the Dehornoy ordering, called {\em Property $S$ (Subword Property)}, fails for the Dehornoy-like ordering of $G_{m,n}$. However, we observe that the Dehornoy-like ordering of $G_{m,n}$ have a slightly weaker property in Theorem \ref{thm:weakS}.

 As an application, by using the Dehornoy-like ordering of $G_{m,n}$,
we construct left orderings having an interesting property: a left-ordering which admits no non-trivial proper convex subgroups. In fact, we observe that almost all normal subgroup of $G_{m,n}$ contains no non-trivial proper convex subgroup in Theorem \ref{thm:exotic}.\\

\textbf{Acknowledgments.} 
This research was supported by JSPS Research Fellowships for Young Scientists.

\section{Dehornoy-like ordering}

Throughout the paper, we always assume that $G$ is a finitely generated group.

\subsection{Dehornoy-like ordering}

Let $\mS= \{s_{1},\ldots,s_{n}\}$ be an ordered finite generating set of $G$. We consider the two submonoids of $G$ defined by $\mS$, the {\em $\mS$-word positive monoid} and the {\em $\sigma(\mS)$-positive monoid}.

A ($\mS$-) positive word is a word on $\mS$. We say an element $g \in G$ is ($\mS$-) {\em word positive} if $g$ is represented by a $\mS$-positive word.
The set of all $\mS$-word positive elements form a submonoid $P_{\mS}$ of $G$, which we call the ($\mS$-){\em word positive monoid}.
The $\mS$-word positive monoid is nothing but a submonoid of $G$ generated by $\mS$.

To define a Dehornoy-like ordering, we introduce slightly different notions.
A word $w$ on $\mS \cup\, \mS^{-1}$ is called {\it $i$-positive} (or, {\em $i(\mS)$-positive}, if we need to indicate the ordered finite generating set $\mS$) if $w$ contains at least one $s_{i}$ but contains no $s_{1}^{ \pm 1},\ldots, s_{i-1}^{\pm 1}, s_{i}^{-1}$. We say an element $g \in G$ is {\em $i$-positive} (or, {\em $i(\mS)$-positive}) if $g$ is represented by an $i$-positive word.
An element $g \in G$ is called $\sigma$-positive ({\em $\sigma(\mS)$-positive}) if $g$ is $i$-positive for some $1 \leq i \leq n$. The notions of {\em $i$-negative} and {\em $\sigma$-negative} are defined in the similar way. The set of $\sigma(\mS)$-positive elements of $G$ forms a submonoid $\Sigma_{\mS}$ of $G$. 
We call the monoid $\Sigma_{\mS}$ the {\em $\sigma$-positive monoid} (or, {\em $\sigma(\mS)$-positive monoid}).

\begin{defn}[Dehornoy-like ordering]
A {\em Dehornoy-like ordering} is a left ordering $<_{D}$ whose positive cone is equal to the $\sigma$-positive monoid $\Sigma_{\mS}$ for some ordered finite generating set $\mS$ of $G$. In this situation, we say $\mS$ {\em defines} a Dehornoy-like ordering.
\end{defn}

To study Dehornoy-like orderings we introduce the following two properties.

\begin{defn}
Let $\mS$ be an ordered finite generating set of $G$. 
\begin{enumerate}
\item We say $\mS$ has {\it Property $A$} ({\em the Acyclic property}) if no $\sigma(\mS)$-positive words represent the trivial element. That is, $\Sigma_{\mS}$ does not contains the identity element $1$. 
\item We say $\mS$ has {\it Property $C$} ({\em the Comparison property}) if every non-trivial element of $G$ admits either $\sigma(\mS)$-positive or $\sigma(\mS)$-negative word expression.
\end{enumerate}
\end{defn}

\begin{prop}
Let $\mS$ be an ordered finite generating set of a group $G$.
Then $\mS$ defines a Dehornoy-like ordering if and only if $\mS$ has both Property $A$ and Property $C$.
\end{prop}
\begin{proof}
Property $C$ implies that $G= \Sigma_{\mS} \cup \Sigma_{\mS}^{-1} \cup \{1\}$, and the Property $A$ implies that $\Sigma_{\mS}$, $\Sigma_{\mS}^{-1}$ and $\{1\}$ are disjoint. Thus, the $\sigma(\mS)$-positive monoid $\Sigma_{\mS}$ satisfies both {\bf LO1} and {\bf LO2}. Converse is clear.
\end{proof}

As we have already mentioned, the definition of Dehornoy-like orderings is motivated from the Dehornoy ordering of the braid groups.

\begin{exam}[Dehornoy ordering of $B_{n}$]
Let us consider the standard presentation of the braid group $B_{n}$,
\[ B_{n} = 
\left \langle 
\sigma_{1},\ldots,\sigma_{n-1}  \;
\begin{array}{|ll}
 \sigma_{i}\sigma_{j}\sigma_{i}=\sigma_{i}\sigma_{j}\sigma_{i} &|i-j|=1 \\
\sigma_{i}\sigma_{j}=\sigma_{j}\sigma_{i}& |i-j|>1 
\end{array}
\right\rangle \]
and $\mS=\{ \sigma_{1},\ldots,\sigma_{n-1}\}$ be the set of the standard generators.
The seminal work of Dehornoy \cite{d1} shows that $\mS$ has both Property $A$ and Property $C$, hence $\mS$ defines a left-ordering of $B_{n}$. The ordering $<_{D}$ defined by $\mS$ is called the {\it Dehornoy ordering}.

Interestingly, there are various proof of Property $A$ and Property $C$, and each proof gives a new characterization of the Dehornoy ordering. A proof of Property $A$ or Property $C$ provides new insights for not only the Dehornoy ordering, but also the braid group itself. See \cite{ddrw} for the theory of the Dehornoy orderings. Moreover, as the author showed, the Dehornoy orderings are also related to the knot theory \cite{i1},\cite{i2}.
\end{exam}

Now we introduce an operation to construct new ordered finite generating sets from an ordered finite generating set, which connects a Dehornoy-like ordering and an isolated ordering.

The {\em twisted generating set} of $\mS$ is an ordered finite generating set $\mA=\mA_{\mS}=\{a_{1},\ldots,a_{n}\}$ where each $a_{i}$ is defined by
\[ a_{i}= (s_{i}\cdots s_{n-1})^{(-1)^{n-i+1}}. \]

An ordered finite generating set $\mD=\mD_{\mS}=\{d_{1},\ldots,d_{n}\}$ whose twisted generating set is equal to $\mS$ is called the {\em detwisted generating set} of $\mS$.
The detwisted generating set $\mD$ is explicitly give as 
\[ d_{i}= \left\{
\begin{array}{ll}
s_{n}^{-1} & i=n\\
s_{i}^{-1}s_{i+1}^{-1} & n-i :\textrm{ even} \\
s_{i}s_{i+1} & n-i : \textrm{ odd} \\
\end{array}
\right.
\]

For each $1\leq i \leq n$, let  $\mS^{(i)}=\{s_{i},s_{i+1},\ldots,s_{n}\}$ and $G^{(i)}_{\mS}$ be the subgroup of $G$ generated by $\mS^{(i)}$. Thus, $\mS^{(i)}$ is an ordered finite generating set of $G^{(i)}_{\mS}$.
We denote the $\mS^{(i)}$-word positive monoid $P_{\mS^{(i)}}$ and the $\sigma(\mS^{(i)})$-positive monoid $\Sigma_{\mS^{(i)}}$ by $P_{\mS}^{(i)}$, $\Sigma_{\mS}^{(i)}$ respectively.
They are naturally regarded as submonoids of $G_{\mS}^{(i)}$.
By definition of the twisted generating set, $\mA_{\mS^{(i)}} = (\mA_{\mS})^{(i)}$. Thus, 
$G^{(i)}_{\mS} = G^{(i)}_{\mA}$ so we will often write $G^{(i)}$ to represent $G^{(i)}_{\mS} = G^{(i)}_{\mA}$.

There is an obvious inclusion for $\sigma(\mS)$-positive and $\mA$-word positive monoids. 

\begin{lem}
\label{lem:inclusion}
Let $\mS$ be an ordered finite generating set and $\mA = \mA_{\mS}$ be the twisted generating set of $\mS$.
Then $\Sigma_{\mS} \cup \Sigma_{\mS}^{-1} \supset P_{\mA} \cup P_{\mA}^{-1}$.
\end{lem}
\begin{proof}
We show $P_{\mA} \subset \Sigma_{\mS} \cup \Sigma_{\mS}^{-1}$. The proof of $P_{\mA}^{-1} \subset \Sigma_{\mS} \cup \Sigma_{\mS}^{-1}$ is similar. Let $g \in P_{\mA}$ and $w$ be an $\mA$-positive word expression of $g$. Put 
\[ i = \min\, \{j \in \{1,2,\ldots,n\} \: | \: w \textrm{ contains the letter } a_{j}\}.\]
Since $a_{i}= (s_{i}s_{i+1}\cdots s_{n})^{(-1)^{n-i+1}}$, $g$ is $\sigma(\mS)$-positive if $(n-i)$ is odd and is $\sigma(\mS)$-negative if $(n-i)$ is even.
\end{proof}

Now we introduce a key property called {\em Property $F$} ({\em the Filtration property}) which allows us to generalize various properties of the Dehornoy ordering for Dehornoy-like orderings.
 
\begin{defn}
Let $\mS=\{s_{1},\ldots,s_{n}\}$ be an ordered finite generating set of $G$ and $\mA=\{a_{1},\ldots,a_{n}\}$ be the twisted generating set of $\mS$.
We say $\mS$ has {\em Property $F$} ({\em the Filtration property}) if
\begin{description}
\item[F] $a_{i} \cdot (P_{\mA}^{(i+1)} )^{-1} \cdot a_{i}^{-1} \subset P_{\mA}^{(i)}$, $a_{i}^{-1}\cdot   ( P_{\mA}^{(i+1)} )^{-1}\cdot a_{i} \subset  P_{\mA}^{(i)}$
\end{description}
hold.
\end{defn}

We say a finite generating set $\mA$ {\em defines} an isolated ordering if the $\mA$-word positive monoid is the positive cone of an isolated ordering $<_{A}$. First we show a Dehornoy-like ordering and an isolated ordering are closely related if we assume Property $F$.

\begin{thm}
\label{thm:DOtoIO}
Let $\mS$ be an ordered finite generating set of $G$ having Property $F$ and $\mA$ be the twisted generating set of $\mS$. Then $\mS$ defines a Dehornoy-like ordering of $G$ if and only if $\mA$ defines an isolated left ordering of $G$.
\end{thm}
\begin{proof}
Let $n$ be the cardinal of the generating set $\mS$.
We prove theorem by induction on $n$. The case $n=1$ is trivial.
General cases follow from the following two claims. 
\begin{claim}
\label{claim:1}
$\mS$ has Property $C$ if and only if $P_{\mA} \cup P_{\mA}^{-1} \cup \{1\} = G$ holds.
\end{claim}

By Lemma \ref{lem:inclusion}, $G= P_{\mA} \cup P_{\mA}^{-1} \cup \{1\} \subset \Sigma_{\mS} \cup \Sigma_{\mS}^{-1} \cup \{1\} = G$.

To show converse, assume that $\mS$ has Property $C$. Let $g \in G$ be a non-trivial element. We assume that $g$ is $\sigma(\mS)$-positive. The case $g$ is $\sigma(\mS)$-negative is proved in a similar way. 

First of all, assume that $g$ has a $k(\mS)$-positive word representative for $k>1$. Then $g \in G^{(2)}$ and $g$ is $\sigma(\mS^{(2)})$-positive. By inductive hypothesis, $g \in P_{\mA}^{(2)} \cup (P_{\mA}^{(2)})^{-1} \cup \{1\} \subset P_{\mA} \cup P_{\mA}^{-1} \cup \{1\} $. 

Thus we assume that $g$ is $1(\mS)$-positive. We also assume that $n$ is even. The case $n$ is odd is similar. Since $s_{1}=a_{1}a_{2}$, by rewriting a $1(\mS)$-positive word representative of $g$ by using the twisted generating set $\mA$, we write $g$ as
\[ g = V_{0} a_{1} V_{1} \cdots a_{1}V_{m} \]
where $V_{i}$ is a word on $\mA^{(2)} \cup (\mA^{(2)})^{-1} \subset G^{(2)}$.
By inductive hypothesis, we may assume that either $V_{i} \in P_{\mA}^{(2)}$ or $V_{i} \in ( P_{\mA}^{(2)} )^{-1}$. If all $V_{i}$ belong to $P_{\mA}^{(2)}$, then $g \in P_{\mA}$.
Assume that some $V_{i}$ belongs to $(P_{\mA}^{(2)} )^{-1}$.
By Property $F$, $a_{1}V_{i} \subset P_{\mA} \cdot a_{1}$ and $V_{i}a_{1} \subset a_{1}\cdot P_{\mA} $, so we can rewrite $g$ so that it belongs to $P_{\mA}$.

\begin{claim}
\label{claim:2}
$\mS$ has Property $A$ if and only if $1 \not \in P_{\mA}$.
\end{claim}
Assume that $1 \not \in P_{\mA}$ and let $g \in G$ be a $\sigma(\mS)$-positive element. If $g$ has a $k(\mS)$-positive word representative for $k>1$, then $g \in G^{(2)}$ so inductive hypothesis shows $g \neq 1$.
Thus we assume $g$ is $1(\mS)$-positive. Assume that $n$ is even. Then as we have seen in the proof of Claim \ref{claim:1}, $g \in P_{\mA}$ so we conclude $g \neq 1$. The case $n$ is odd, and the case $g$ is $\sigma(\mS)$-negative are proved in a similar way.
Converse is obvious from Lemma \ref{lem:inclusion}.
\end{proof}

Thus, we obtain a new criterion for an existence of isolated orderings.

\begin{cor}
Let $G$ be a left-orderable group. If $G$ has a Dehornoy-like ordering having Property $F$, then $G$ also has an isolated left ordering.
\end{cor}

Theorem \ref{thm:DOtoIO} is motivated from the construction of the Dubrovina-Dubrovin orderings, which are left-ordering obtained by modifying the Dehornoy ordering.

\begin{exam}[Dubrovina-Dubrovin ordering]
Let  $\mA = \{ a_{1},\ldots,a_{n-1}\}$ be the twisted generating set of the standard generating set $\mS = \{\sigma_{1},\ldots,\sigma_{n-1}\}$ of the braid group $B_{n}$. 
$\mS$ has the property $F$, so, the submonoid $\mA$ defines an isolated left-ordering which is known as the {\it Dubrovina-Dubrovin ordering} $<_{DD}$ \cite{dd}.
\end{exam}

\subsection{Property of Dehornoy-like orderings}

In this section we study fundamental properties of Dehornoy-like orderings and isolated orderings derived from the Dehornoy-like orderings.

Let $\mS= \{ s_{1},\ldots,s_{n} \} (n>1)$ be an ordered finite generating set of a group $G$ which defines a Dehornoy-like ordering $<_{D}$ and $\mA$ be the twisted generating set of $\mS$.

We begin with recalling standard notions of left orderable groups.
A left-ordering $<_{G}$ of $G$ is called {\it discrete} if there is the $<_{G}$-minimal positive element. Otherwise, $<_{G}$ is called {\it dense}.  $<_{G}$ is called a {\it Conradian ordering} if $fg^{k}>_{G} g$ holds for all $<_{G}$-positive $f,g \in G$ and $k\geq 2$. It is known that in the definition of Conradian orderings it is sufficient to consider the case $k=2$. That is, $<_{G}$ is Conradian if and only of $fg^{2} >_{G} g$ holds for all $<_{G}$-positive elements $f,g \in G$ (See \cite{n1}). 

A subgroup $H$ of $G$ is {\it $<_{G}$-convex} if $h <_{G} g <_{G} h'$ for $h,h'\in H$ and $g \in G$, then $g \in H$ holds. $<_{G}$-convex subgroups form a chain. That is, for $<_{G}$-convex subgroups $H, H'$, either $H \subset H'$ or $H' \subset H$ holds.
The {\it $<_{G}$-Conradian soul} is the maximal (with respect to inclusions) $<_{G}$-convex subgroup of $G$ such that the restriction of $<_{G}$ is Conradian.

First of all, we observe that a Dehornoy-like ordering have the following good properties with respect to the restrictions.

\begin{prop}
\label{prop:convexD}
Let $\mS=\{s_{1},\ldots,s_{n}\}$ be an ordered finite generating set of $G$ which defines a Dehornoy-like ordering $<_{D}$.
\begin{enumerate}
\item For $1 \leq i \leq n$, $\mS^{(i)}$ defines a Dehornoy-like ordering $<_{D}^{(i)}$ of $G^{(i)}$. Moreover, the restriction of the Dehornoy-like ordering $<_{D}$ to $G^{(i)}$ is equal to the Dehornoy-like ordering $<_{D}^{(i)}$.
\item For $1\leq i \leq n$, the subgroup $G^{(i)}$ is $<_{D}$-convex.
In particular, $<_{D}$ is discrete, and the minimal $<_{D}$-positive element is $s_{n}$.
\item If $H$ is a $<_{D}$-convex subgroups of $G$, then $H=G^{(i)}$ for some $1 \leq i \leq n$.
\end{enumerate}
\end{prop}
\begin{proof}
Since $\mS$ has Property $A$, $\mS^{(i)}$ has Property $A$. Assume that $\mS^{(i)}$ does not have Property $C$, so there is an element $g \in G^{(i)} -\{1\}$ which is neither $\sigma(\mS^{(i)})$- positive nor $\sigma(\mS^{(i)})$-negative.
Assume that $1 <_{D} g$, so $g$ is represented by a $\sigma(\mS)$-positive word $W$. The case $1>_{D} g$ is similar.
Since $g \in G^{(i)}$ we may find a word representative $V$ of $g$ which consists of the alphabets in $\mS^{(i)} \cup {\mS^{(i)} }^{-1}$. Then $V^{-1}W$ is $\sigma(\mS)$-positive word which represents the trivial element, so this contradicts the fact that  $\mS$ has Property $A$. Thus, $\mS^{(i)}$ has Property $C$, hence $\mS^{(i)}$ defines a Dehornoy-like ordering $<_{D}^{(i)}$. 
Now the Property $A$ and Property $C$ of $\mS^{(i)}$ implies $\Sigma_{\mS^{(i)}} = \Sigma_{\mS} \cap G^{(i)}_{\mS}$, so  $<_{D}^{(i)}$ is equal to the restriction of $<_{D}$ to $G^{(i)}_{\mS}$.

Next we show $G^{(i)}$ is $<_{D}$-convex. Assume that $1<_{D} h <_{D} g$ hold for $g \in G^{(i)}$ and $h \in G$.
If $h$ is $j(\mS)$-positive for $j < i$, then $g^{-1}h$ is also $j(\mS)$-positive, so $g<_{D} h$. This contradicts the assumption, so $h$ must be $j(\mS)$-positive for $j\geq i$. This implies $h \in G^{(i)}$, so we conclude $G^{(i)}$ is $<_{D}$-convex. 

To show there are no $<_{G}$-convex subgroups other than $G^{(i)}$, it is sufficient to show if $H \supset G^{(2)}$ then $H = G^{(2)}$ or $G^{(1)} = G$.
Assume that $H \neq G^{(2)}$, hence $H$ contains an element $g$ in $G - G^{(2)}$. Let us take such $g$ so that $1<_{D} g$. Then $g$ must be $1(\mS)$-positive, hence we may write $g= h s_{1} P$ where $h \in G^{(2)}$ and $P>_{D} 1$. Then we have
\[ 1 <_{D} hs_{1} \leq_{D} hs_{1}P = g\]
 Since $H$ is convex, this implies $hs_{1} \in H$. Since $h \in G^{(2)} \subset H$, we conclude $s_{1} \in H$ hence $H=G$.
\end{proof}

From now on, we will always assume that $\mS$ has Property $F$, hence $\mA$ defines an isolated left ordering $<_{A}$. First of all we observe that $<_{A}$ also has the same properties as we have seen in Proposition \ref{prop:convexD}.

\begin{prop}
\label{prop:convexA}
Let $\mA=\{a_{1},\ldots,a_{n}\}$ be the twisted generating set of $\mS$ which defines an isolated left ordering $<_{A}$.
\begin{enumerate}
\item For $1 \leq i \leq n$, $\mA^{(i)}$ defines an isolated ordering $<_{A}^{(i)}$ of $G^{(i)}$. Moreover, the restriction of the isolated ordering $<_{A}$ to $G^{(i)}$ is equal to the isolated ordering $<_{A}^{(i)}$.
\item For $1\leq i \leq n$, the subgroup $G^{(i)}$ is $<_{A}$-convex.
In particular, $<_{A}$ is discrete, and the minimal $<_{A}$-positive element is $a_{n}$.
\item If $H$ is an $<_{A}$-convex subgroup of $G$, then $H=G^{(i)}$ for some $1 \leq i \leq n$.
\end{enumerate}
\end{prop}
\begin{proof}
The proofs of (1) and (3) are similar to the case of Dehornoy-like orderings.
To show (2), assume that $1<_{A} h <_{A} g$ hold for $g \in G^{(i)}$ and $h \in G$.
If $h \neq P_{\mA}^{(i)}$, then we may write $h$  as $h= h' a_{j} w$
where $h' \in P_{\mA}^{(j+1)}$ and $w \in P_{\mA}^{(j)} \cup \{1\}$ for $j<i$.
If $g^{-1}h' \in P_{\mA}^{(j+1)}$, then $g^{-1}h = (g^{-1}h') a_{j}w >_{A} 1 $.
If $g^{-1}h' \in (P_{\mA}^{(j+1)})^{-1}$, then by Property $F$, $(g^{-1}h')a_{j} >_{A} 1$ so $g^{-1}h = [(g^{-1}h') a_{j}]w >_{A} 1 $.
Therefore in both cases, $g^{-1}h>_{A} 1$, it is a contradiction. 

\end{proof}

To deduce more precise properties, we observe the following simple lemma.

\begin{lem}
\label{lem:key}
If $a_{n-1}a_{n}a_{n-1}\neq a_{n}$, then $a_{n-1}a_{n}^{-1}a_{n-1}^{-1}, a_{n-1}^{-1}a_{n}a_{n-1} \in P_{\mA}^{(n-1)} - P_{\mA}^{(n)}$.
\end{lem}
\begin{proof}
To make notation simple, we put $p=a_{n-1}$ and $q=a_{n}$.

By Property $F$, $pq^{-1}p^{-1}, p^{-1}q^{-1}p \in P_{\mA}^{(n-1)}$.
We show $pq^{-1}p^{-1} \neq P_{\mA}^{(n)}$. The proof of $p^{-1}q^{-1}p \neq P_{\mA}^{(n)}$ is similar. Assume that $pq^{-1}p^{-1}=q^{k}$ for $k>0$. Since we have assumed that $qpq \neq q$, $k>1$.
Then
\[ q= p^{-1}(pq^{-1}p^{-1})^{-1}p = p^{-1} q^{-k} p = (p^{-1}q^{-1}p)^{k},\]
so we have $1 <_{A} p^{-1}q^{-1}p <_{A} q$. This contradicts Proposition \ref{prop:convexA} (2), the fact that $q=a_{n}$ is the minimal $<_{A}$-positive element. 
\end{proof}

Next we show that in most cases the Dehornoy-like ordering $<_{D}$ is not isolated, so it makes a contrast to the isolated ordering $<_{A}$. 
For $g \in G$ and a left ordering $<$ of $G$ whose positive cone is $P$, we define $<_{g} = < \cdot \,g$ as the left ordering defined by the positive cone $P\cdot g$. Thus, $x <_{g} x'$ if and only if $xg < x'g$. This defines a right action of $G$ on $\LO(G)$.
Two left orderings are said to be {\em conjugate} if they belong to the same $G$-orbit.

\begin{thm}
\label{thm:converge}
If $a_{n-1}a_{n}a_{n-1}\neq a_{n}$, then the Dehornoy-like ordering $<_{D}$ is an accumulation point of the set of its conjugates $\{ <_{D} \cdot \,g \}_{g \in G}$. Thus, $<_{D}$ is not isolated in $\LO(G)$, and the $\sigma(\mS)$-positive monoid is not finitely generated. 
\end{thm}

\begin{proof}

Our argument is generalization of Navas-Wiest's criterion \cite{nw}.
As in the proof of Lemma \ref{lem:key}, we put $p=a_{n-1}$ and $q=a_{n}$ to make notation simple.

We construct a sequence of left orderings $\{<_{n}\}$ so that $\{<_{n}\}$ non-trivially converge to $<_{D}$ and that each $<_{n}$ is conjugate to $<_{D}$. Here the word non-trivially means that $<_{n} \neq <_{D}$ for sufficiently large $n>0$.

Let $<_{n} = <_{D} \cdot \;(q^{n}p)$. Thus, $1<_{n} g$ if and only if $1<_{D}(q^{n}p)^{-1} g(q^{n}p)$. First we show the orderings $<_{n}$ converge to $<_{D}$ for $n \rightarrow \infty$.
By definition of the topology of $\LO(G)$, it is sufficient to show that for an arbitrary finite set of $<_{D}$-positive elements $c_{1},\ldots,c_{r}$, $1 <_{N} c_{i} $ holds for sufficiently large $N>0$.

If $c_{i} \not \in G^{(n-1)}$, then $1 <_{n} c_{i}$ for all $n$. Thus, we assume $c_{i} \in G^{(n-1)}$.

First we consider the case $c_{i} \not \in G^{(n)}$. Then $c_{i}$ is $(n-1)(\mS)$-positive, hence by using generators $\{p,q\}$, $c_{i}$ is written as $ c_{i}= q^{m} p w$ where $w \in \Sigma_{\mS}^{(n-1)}$ and $m \in \Z$.
For $k>m$, by Property $F$ $p^{-1}q^{m-k}p \in P_{\mA}^{(n-1)}$. So, if we take $k>m$, then  
\[ (q^{k}p)^{-1} c_{i} (q^{k}p) = (p^{-1}q^{m-k}p) w q^{k}p \]
is $(n-1)(\mS)$-positive. So $1<_{k} c_{i}$ for $k>m$.

Next assume that $c_{i} \in G^{(n)}$, so $c_{i}= q^{-m}$ $(m>0)$. Then 
$(q^{k}p)^{-1} c_{i} (q^{k}p) = p^{-1}q^{-m}p$. By Lemma \ref{lem:key}, $p^{-1}q^{-m}p \in P_{\mA}^{(n-1)} - P_{\mA}^{(n)}$, so $p^{-1}q^{-m}p$ is $(n-1)(\mS)$-positive. Therefore $1 <_{D} p^{-1}q^{-m}p $, hence $1<_{k} c_{i}$ for all $k>1$.

To show the convergent sequence $\{<_{n}\}$ is non-trivial, we observe that the minimal positive element of the ordering $P_{n}$ is $(q^{k}p)^{-1} q (q^{k}p) = p^{-1}qp$. From the assumption, $p^{-1}qp$ is not identical with $q$, the minimal positive element of the ordering $<_{D}$. Thus, $<_{n}$ are different from the ordering $<_{D}$.
\end{proof}

It should be mentioned that our hypothesis $a_{n-1}a_{n}a_{n-1}\neq a_{n}$ is really needed.
Let us consider the Klein bottle group $G = \langle s_{1},s_{2} \: | \: s_{2}s_{1}s_{2}=s_{1} \rangle$. It is known that $\mS=\{s_{1},s_{2}\}$ defines a Dehornoy-like ordering $<_{D}$ of $G$ (See \cite{n2} or the proof of Theorem \ref{thm:order} in Section 3.1, which is valid for the Klein bottle group case, $(m,n)=(2,2)$).
However, since $G$ has only finitely many left orderings, $<_{D}$ must be isolated. Observe that for the twisted generating set $\mA=\{a_{1},a_{2}\}$ of $\mS$, the Klein bottle group has the same presentation $G = \langle a_{1},a_{2} \: | \: a_{2}a_{1}a_{2}=a_{1} \rangle$.  

Finally we determine the Conradian soul of $<_{D}$ and $<_{A}$.

\begin{thm}[Conradian properties of Dehornoy-like and isolated orderings]
Let $\mS=\{s_{1},\ldots,s_{n}\} (n>1)$ be an ordered finite generating set of a group $G$ which defines a Dehornoy-like ordering $<_{D}$. Assume that $\mS$ has Property $F$ and let $\mA=\{a_{1},\ldots,a_{n}\}$ be the twisted generating set of $\mS$. Let $<_{A}$ be the isolated ordering defined by $\mA$.
If $a_{n-1}a_{n}a_{n-1}\neq a_{n}$, then two orderings $<_{D}$ and $<_{A}$ have the following properties. 
\begin{enumerate}
\item $<_{D}$ is not Conradian. Thus, the $<_{D}$-Conradian soul is $G^{(n)}$, the infinite cyclic subgroup generated by $s_{n}$.
\item $<_{A}$ is not Conradian. Thus, the $<_{A}$-Conradian soul is $G^{(n)}$, the infinite cyclic subgroup generated by $a_{n}$.
\end{enumerate}
\end{thm}
\begin{proof}
As in the proof of Lemma \ref{lem:key}, we put $p=a_{n-1}$ and $q=a_{n}$. To prove theorem it is sufficient to show for $n>2$, $<_{D}$ and $<_{A}$ are not Conradian, since by Proposition \ref{prop:convexD} and Proposition \ref{prop:convexA} if $H$ is a $<_{D}$- or $<_{A}$- convex subgroup, then $H=G^{(i)}$ for some $i$.
First we show $<_{A}$ is not Conradian.
By Lemma \ref{lem:key}, every $\mA$-word positive representative of $p^{-1}q^{-1}p$ contains at least one $p$, so we put $p^{-1}qp=N p^{-1}q^{-k}$ where $N \leq_{A} 1$ and $k\geq 0$.
Then we obtain an inequality
\[ (q^{k}p)^{-1}q(q^{k}p)^{2} = (p^{-1} q p) q^{k} p = N \leq_{A} 1 \]
hence $<_{A}$ is not Conradian. 

To show $<_{D}$ is not Conradian, we observe 
\begin{eqnarray*}
 (pq^{k+1}pq^{k+1})^{-1}(pq^{k+2})(pq^{k+1}pq^{k+1})^{2}& = & q^{-(k+1)}(p^{-1}qp) q^{k+1}pq^{k+1}pq^{k+1}pq^{k+1}  \\
 & = & q^{-(k+1)}N (p^{-1}q)pq^{k+1}pq^{k+1}pq^{k+1} \\
 & = & \cdots \\
 & = &q^{-(k+1)}N^{4}p^{-1}q\\
 & <_{D} & 1.
 \end{eqnarray*}
\end{proof}

\begin{ques}
In our study of Dehornoy-like ordering, we assumed somewhat artificial conditions, such as Property $F$ or $a_{n-1}a_{n}a_{n-1}\neq a_{n}$. Are such assumptions really needed ? that is,  
\begin{enumerate}
\item {\em If $\mS$ defines a Dehornoy like ordering of $G$, then does $\mS$ have Property $F$ ?}
\item {\em If $G$ has infinitely many left orderings and $\mS$ defines a Dehornoy-like ordering, then does $a_{n-1}a_{n}a_{n-1}\neq a_{n}$ hold for $a_{n-1},a_{n}$ in the twisted generating set $\mA$ of $\mS$ ?}
\end{enumerate}
It is likely that the above questions have affirmative answers, hence the relationships between Dehornoy-like orderings and isolated orderings are quite stronger than stated in this paper.
\end{ques}

\section{Isolated and Dehornoy-like ordering on $\Z *_{\Z} \Z$}

 In this section we construct explicit examples of Dehornoy-like and isolated left orderings of the group $\Z *_{\Z} \Z$ and study more detailed properties. 
 
\subsection{Construction of orderings}

First we review the notations.
Let $G=\Z * _{\Z} \Z$ be the amalgamated free product of two infinite cyclic groups.
As we mentioned such a groups are presented as 
\[ G_{m,n} = \langle x,y \: | \: x^{m}=y^{n} \rangle. \;\;\;\;\;(m \geq n)\] 
and we will always assume $(m,n) \neq (2,2)$.

We consider an ordered generating set $\mS=\{s_{1}=xyx^{-m+1}, s_{2}=x^{m-1}y^{-1}\}$ and its twisted generating set  $\mA =\{a = x,b = yx^{-m+1}\}$. Using $\mS$, $\mA$, the group $G_{m,n}$ is presented as 
\begin{eqnarray*}
 G_{m,n} &  = & \langle s_{1},s_{2} \: | \: s_{2}s_{1}s_{2}=
( (s_{1}s_{2})^{m-2}s_{1})^{n-1} \rangle \\
 & = & \langle a,b \: | \: (ba^{m-1})^{n-1}b=a \rangle
\end{eqnarray*}
respectively. In this section we present a new family of Dehornoy-like and isolated left orderings.

\begin{thm}
Let $\mS, \mA$ be the ordered finite generating sets of $G_{m,n}$ as above. 
\label{thm:order}
\begin{enumerate}
\item $\mS$ defines a Dehornoy-like ordering $<_{D}$.
\item $\mA$ defines an isolated left ordering $<_{A}$.
\end{enumerate}
\end{thm}

By the presentation of $G$, it is easy to see that $\mS$ have Property $F$ and $bab\neq b$ if $(m,n) \neq (2,2)$. Thus from general theories developed in Section 2, we obtain various properties of $<_{D}$ and $<_{A}$.

\begin{cor}
Let $<_{D}$ be the Dehornoy-like ordering of $G=G_{m,n}$ and $<_{A}$ be the isolated ordering in Theorem \ref{thm:order}.
\begin{enumerate}
\item If $H$ is a $<_{D}$-convex subgroup, then $H =\{1\}$ or $\langle s_{2} \rangle$ or $G$.
\item If $H$ is a $<_{A}$-convex subgroup, then $H =\{1\}$ or $\langle b \rangle$ or $G$.
\item The Conradian soul of $<_{D}$ is $G^{(2)} = \langle s_{2} \rangle$.
\item The Conradian soul of $<_{A}$ is $G^{(2)} = \langle b \rangle$.
\item $<_{D}$ is an accumulation point of the set of its conjugate $\{ <_{D} \cdot \:g \}_{g \in G}$. Thus, $<_{D}$ is not isolated in $\LO(G)$ and the $\sigma(\mS)$-positive monoid is not finitely generated. 
\end{enumerate}
\end{cor}

Before giving a proof, first we recall the structure of the group $G_{m,n}$.
Let $Z_{m,n} = \Z_{m} * \Z_{n} = \langle X,Y \; | \; X^{m}=Y^{n}=1 \rangle$, where $\Z_{m}$ is the cyclic group of order $m$ and let $\pi: G_{m,n} \rightarrow Z_{m,n}$ be a homomorphism  defined by $\pi(x)= X$, $\pi(y)= Y$. The kernel of $\pi$ is an infinite cyclic group generated by $x^{m}=y^{n}$ which is the center of $G_{m,n}$. Thus, we have a central extension
\[ 1 \rightarrow \Z \rightarrow G_{m,n} \rightarrow Z_{m,n} \rightarrow 1.\]

We describe an action of $Z_{m,n}$ on $S^{1}$ which is used to prove Property $A$.
Let $T=T_{m,n}$ be the Bass-Serre tree for the free product $Z_{m,n}=\Z_{m}*\Z_{n}$. That is,
$T_{m,n}$ is a tree whose vertices are disjoint union of cosets $\Z_{m,n} \slash \Z_{m} \coprod \Z_{m,n} \slash \Z_{n}$ and edges are $Z_{m,n}$. Here an edge $g \in Z_{m,n}$ connects two vertices $g \Z_{m}$ and $g\Z_{n}$. (See Figure \ref{fig:tree} left for example the case $(m,n)=(4,3)$). 

In our situation, the Bass-Serre tree $T$ is naturally regarded as a planer graph. More precisely, we regard $T$ as embedded into the hyperbolic plane $\mathbb{H}^{2}$. Now $X$ acts on $T_{m,n}$ as a rotation centered on $P$, and $Y$ acts on $T$ as a rotation centered on $Q$. This defines an faithful action of $Z_{m,n}$ on $T$. Let $E(T)$ be the set of the ends of $T$, which is identified with the set of infinite rays emanating from a fixed base point of $T$. The end of tree $E(T)$ is a Cantor set, and naturally regarded as a subset of the points at infinity $S_{\infty}^{1}$ of $\mathbb{H}^{2}$.
The action of $Z_{m,n}$ induces a faithful action on $E(T)$, and this action extends an action on $S^{1}$. We call this action the {\em standard action} of $Z_{m,n}$.

The standard action is easy to describe since $X$ and $Y$ act as rotations of the tree $T$.
$X$ acts on $S^{1}$ so that it sends an interval $[p_{i},p_{i+1}]$ to the adjacent interval $[p_{i+1},p_{i+2}]$ (here indices are taken modulo $m$), and $Y$ acts on $S^{1}$ so that it sends an interval $[q_{i},q_{i+1}]$ to $[q_{i+1},q_{i+2}]$ (here indices are taken modulo $n$).
See Figure \ref{fig:tree} right. More detailed description of the set of ends $E(T)$ and the standard action will be given in next section.

\begin{figure}[htbp]
 \begin{center}
  \includegraphics[width=100mm]{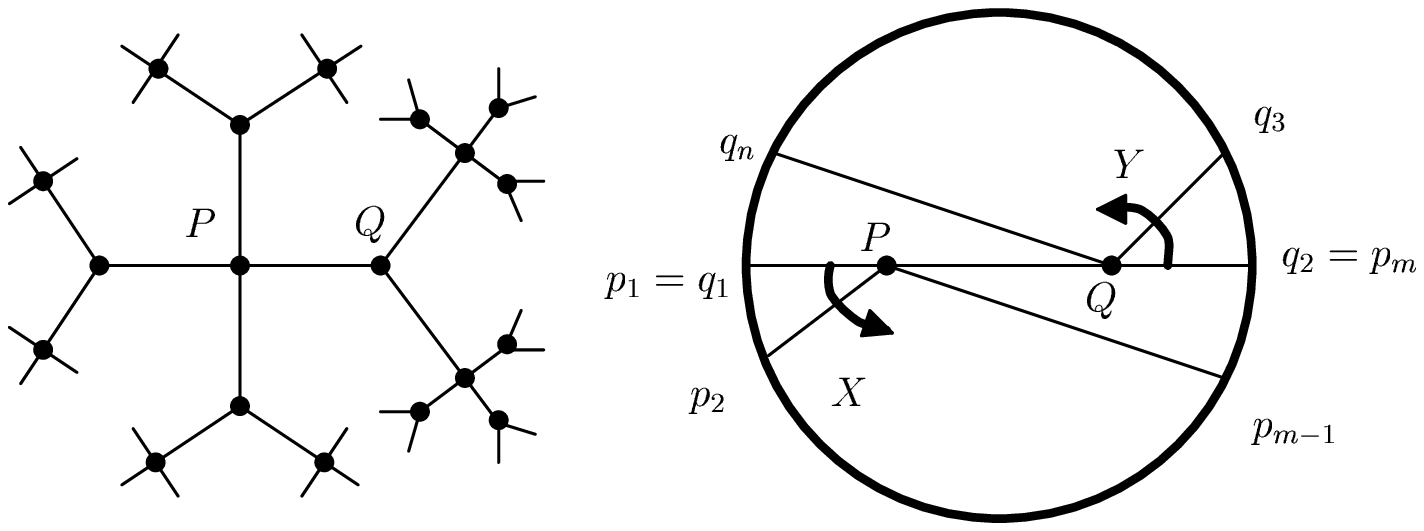}
 \end{center}
 \caption{Bass-Serre Tree and action of $Z_{m,n}$ on $S^{1}$}
 \label{fig:tree}
\end{figure}

Using the standard action, we show the Property $A$ for $\mS$, which is equivalent to the following statement by Claim \ref{claim:2}. 

\begin{lem}
\label{lem:A}
If $g \in P_{\mA}$, then $g \neq 1$.
\end{lem}
\begin{proof}

Let $g \in P_{\mA}$ and put $A = \pi(a) = X$ and $B=\pi(b)=YX^{-m+1}=YX$.
If $g \in \textrm{Ker}\, \pi$, that is, $g=a^{mN}$ for $N>0$, then $g \neq 1$ is obvious so we assume $g \neq a^{mN}$.
Let us put  
\[ \pi(g) = A^{s_{k}} B^{r_{k}} \cdots A^{s_{1}}B^{r_{1}}.\]
where $0< s_{i} < m$ $(i < k)$, $0 \leq s_{k} < m$ and $0< b_{i}$ $(i>1)$, $0 \leq b_{1}$.

First observe that the dynamics of $B$, $A$ and $(BA^{m-1})^{i}B$ are given by the following formulas.
\begin{eqnarray*}
B[p_{m},p_{m-1}] & =& YX[p_{m},p_{m-1}] =Y[p_{1},p_{m}] = Y[q_{1},q_{2}]=[q_{2},q_{3}]\subset[p_{m},p_{1}]\\
A^{i}[p_{m},p_{1}] & \subset &[p_{m},p_{m-1}] \;\;(i \neq m-1)\\
(BA^{m-1})^{i}B[p_{m},p_{m-1}] & = &(YX X^{m-1})^{i}YX[p_{m},p_{m-1}] =Y^{i+1}[p_{1},p_{m}] \subset Y^{i+1}[q_{1},q_{2}] \\
& \subset&  [q_{2+i},q_{3+i}] \subset [p_{m},p_{1}].
\end{eqnarray*}
Here $[p_{m},p_{m-1}]$ represents the interval $[p_{m},p_{1}]\cup[p_{1},p_{2}]\cup \cdots \cup [p_{m-2},p_{m-1}]$.
Since $(BA^{m-1})^{n-1}B =A$, we can assume that the above word expression does not contains a subword of the form $(BA^{m-1})^{n-1}B$.
Thus, by using of the above formulas repeatedly, we conclude 
\[
\left\{
\begin{array}{ll}
\pi(g) [p_{r_{k}-1},p_{r_{k}}] =  [p_{r_{k}},p_{r_{k}+1}]  & (s_{1}\neq 0, r_{k} \neq 0)\\
\pi(g) [p_{m},p_{1}] =  [p_{r_{k}},p_{r_{k}+1}] &  (s_{1}= 0, r_{k} \neq 0)\\
\pi(g) [p_{1},p_{2}] = [p_{m},p_{1}] &  (s_{1} \neq 0, r_{k} = 0)
\end{array}
\right.
\]
So we conclude $\pi(g) \neq 1$ if $s_{1}\neq 0$ or $r_{1} \neq 0$.
If $s_{1}=r_{k}=0$, we need more careful argument to treat the case the word $\pi(g)$ contains a subword of the form $(BA^{m-1})^{i}$.
Let us write 
\[ \pi(g) = B^{s_{k}-1} (BA^{m-1})^{i} \underbrace{B \cdots A^{r_{1}}}_{*} \]
where we take $i$ the maximal among such description of the word $\pi(g)$. That is, the prefix of the word $*$ is not $BA^{m-1}$. 
Then for $i \neq 0$,
\[ \pi(g) [q_{i},q_{i+1}] \subset B^{s_{k}-1}(BA^{m-1})^{i}[p_{m},p_{1}] \subset  B^{s_{k}-1}[p_{m},p_{1}]. \]
Thus if $s_{k} \neq 1$, then 
$\pi(g)[q_{n},q_{1}] \subset [q_{2},q_{3}]$.
and if $s_{k}=1$, then $\pi(g)[q_{2},q_{3}] =[q_{2+i},q_{3+i}]$. Thus we proved that in all cases $\pi(g)$ acts on $S^{1}$ non trivially, hence $g \neq 1$.
\end{proof}

Next we show Property $C$, which is equivalent to the following statement according to Claim \ref{claim:1}.

\begin{lem}
\label{lem:B}
$P_{\mA} \cup \{1\} \cup P_{\mA}^{-1} = G$ .
\end{lem}
\begin{proof}

Let $g \in G$ be a non-trivial element. 

Since $a^{m}=(ba^{m-1})^{n}=(a^{m-1}b)^{n}$ is central, we may write $g$ as
\[ g= a^{mM} a^{s_{k}} b^{r_{k}} \cdots a^{s_{1}}b^{r_{1}}. \]
where $0< s_{i} < m$ $(i < k)$, $0 \leq s_{k} < m$ and $0< b_{i}$ $(i>1)$, $0 \leq b_{1}$.

Among such word expressions, we choose the word expression $w$ so that $k$ is minimal.
If $mM+s_{k} \geq 0$, then $g \in P_{\mA}$. 
So we assume that $mM+s_{k} <0$ and we prove $g \in P_{\mA}^{-1}$ by induction on $k$.
The case $k=1$ is a direct consequence of Property $F$.

First observe that from a relation $a^{-1}b = (a^{m}b)^{-n}$, we get a relation
\[ a^{-1}b^{r} = [(a^{m}b)^{-(n-1)} b^{-1}a^{-m+1}]^{r-1} (a^{m}b)^{-n} \]
for all $r>0$.
Thus, by applying this relation, $g$ is written as
\[ g= X (a^{m}b)^{-n} \cdot  a^{s_{2}}b^{r_{2}} \cdots \]
for some $X \in P_{\mA}^{-1}$.
Unless $s_{k-1}=s_{k-2}= \cdots = s_{k-n} = m$ and $r_{k-1}= r_{k-2} = \cdots = r_{k-n}=1$, by reducing this word expression, we obtain a word expression of the form 
\[ g= X' a^{-1} b^{r_{i}'} a^{s_{i-1}} \cdots \]
where $X' \in P_{\mA}^{-1}$ and $i<k$.
By inductive hypothesis, $ a^{-1} b^{r_{i}'}a^{s_{i-1}} \cdots \in P_{\mA}^{-1}$, hence we conclude $g \in P_{\mA}^{-1}$.
 
Now assume $s_{k-1}=s_{k-2}= \cdots = s_{k-n} = m$ and $r_{k-1}= r_{k-2} = \cdots = r_{k-n}=1$. Since $(a^{m}b)^{n}= b^{-1}a$, by replacing the subword $(a^{m}b)^{n}$ with $b^{-1}a$ and canceling $b^{-1}$ we obtain another word representative 
\[ g= a^{(m+1) M } a^{s_{k}}b^{r_{k}-1} a^{s_{k-n-1} +1} b^{r_{k-n-1}} \cdots \]
 which contradicts the assumption that we have chosen the first word representative of $g$ so that $k$ is minimal.
\end{proof}

These two lemmas and Theorem \ref{thm:DOtoIO} prove Theorem \ref{thm:order}.

\begin{ques}
Theorem \ref{thm:order} provides a negative answer to the Main question rasied by Navas in \cite{n2}, which concerns the charactarization of groups having an isolated ordering defined by two elements. Now we ask the following refined version of the question:
\begin{enumerate}
\item[($1$)] {\em If a group $G$ has a Dehornoy-like ordering defined by an ordered generating set of cardinal two, then $G=G_{m,n}$ for some $m,n>1$ ?}
\item[($1'$)] {\em If a group $G$ has a Dehornoy-like ordering defined by an ordered generating set of cardinal two having Property F, then $G=G_{m,n}$ for some $m,n>1$ ?}
\item[($2$)] {\em If a group $G$ has an isolated ordering whose positive cone is generated by two elements, then $G=G_{m,n}$ for some $m,n>1$ ?}
\end{enumerate}
These questions also concern the Question 1, since an affirmative answer for (1) yields an 
affirmative answer of Question 1 for the case $n=2$. 

We also mention that to find a group $G \neq B_{n}$ with a Dehornoy-like ordering defined by an ordered generating set of cardinal more than two is an open problem. 
In particular, does a natural candidate, a group of the form $(\cdots (\Z*_{\Z} \Z)*_{\Z}\Z) \cdots )*_{\Z} \Z$ have a Dehornoy-like ordering ?
\end{ques}

\subsection{Dynamics of the Dehornoy-like ordering of $\Z*_{\Z} \Z$}

As is well-known, if $G$ faithfully acts on the real line $\R$ as orientation preserving homeomorphisms, then one can construct a left-ordering of $G$ as follows.
Let $\Theta: G \hookrightarrow \textrm{Homeo}_{+}(\R)$ be a faithful left action of $G$ and choose a dense countable sequence $I= \{x_{i}\}_{i=1,2,\ldots}$ of $\R$.
For $g,g' \in G$ we define $g <_{I} g'$ if there exists $j$ such that $g(x_{i}) = g'(x_{i})$ for $i <j$ and $g(x_{j}) <_{\R} g'(x_{j})$, where $<_{\R}$ be the standard ordering of $\R$ induced by the orientation of $\R$. 
We say the ordering $<_{I}$ is defined by the action $\Theta$ and the sequence $I$.
Assume that the stabilizer of a finite initial subsequence $\{ x_{1},\ldots,x_{k} \}$ is trivial. Then to define the ordering, we do not need to consider the rest of sequence $\{x_{k+1},x_{k+2},\ldots\}$. In such case, we denote the ordering $<_{I}$ by $<_{\{x_{1},\ldots,x_{k}\}}$ and call the {\em ordering defined by} $\{x_{1},\ldots,x_{k}\}$ (and the action $\Theta$).
 
Conversely, one can construct an orientation-preserving faithful action of $G$ on the real line $\R$ from a left-ordering of $G$ if $G$ is countable. See \cite{n1} for details.

In this section we give an alternative definition of the Dehornoy-like ordering $<_{D}$ of $G_{m,n}$ by using the dynamics of $G_{m,n}$. Recall that $G_{m,n}$ is a central extension of $Z_{m,n}$ by $\Z$.
By lifting the standard action of $Z_{m,n}$ on $S^{1}$, we obtain a faithful orientation-preserving action of $G_{m,n}$ on the real line. We call this action the {\it standard action} of $G_{m,n}$ and denote by $\Theta: G_{m,n} \rightarrow \textrm{Homeo}_{+}(\R)$.

We give a detailed description of the action of $G_{m,n}$ on $\R$.
First of all, we give a combinatorial description of the end of the tree $T=T_{m,n}$.
 
Let us take a basepoint $*$ of $T$ as the midpoint of $P$ and $Q$.
For each edge of $T$, we assign a label as in Figure \ref{fig:label}. 
Let $e$ be a point of $E(T)$, which is represented by an infinite ray $\gamma_{e}$ emanating from $*$.
Then by reading a label on edge along the infinite path $\gamma_{e}$, $\gamma_{e}$ is encoded by an infinite sequence of integers $\pm l_{1}l_{2}\cdots$. 

\begin{figure}[htbp]
 \begin{center}
  \includegraphics[width=110mm]{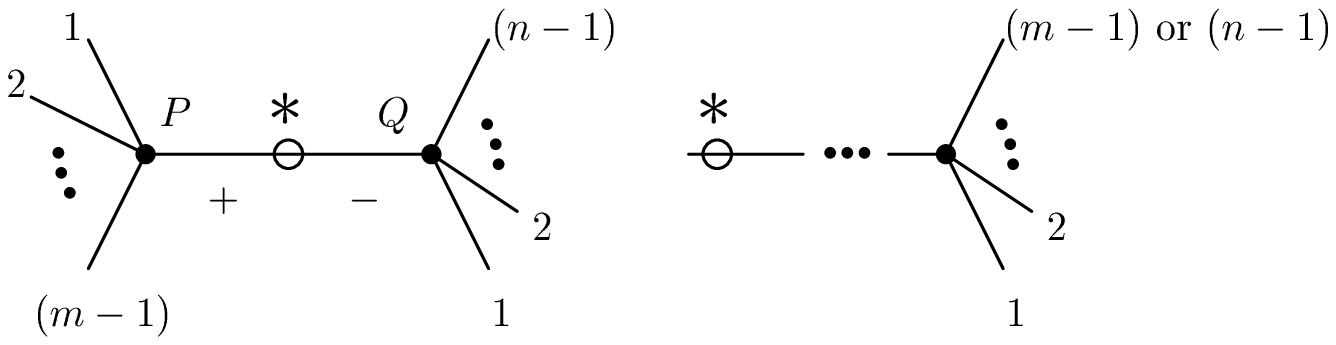}
 \end{center}
 \caption{Labeling of edge}
 \label{fig:label}
\end{figure}

Now let $p: \R \rightarrow S^{1}$ be the universal cover, and $\tE(T) = p^{-1}(E(T)) \subset \R$. Then the standard action $G_{m,n}$ preserves $\tE(T)$.
A point of $\tE(T)$ is given as the sequence of integers $(N; \pm l_{1}l_{2}\cdots)$ where $N \in \Z$.
Observe that the set of such a sequence of integers has a natural lexicographical ordering.
This ordering of $\tE(T)$ is identical with the ordering induced by the standard ordering $<_{\R}$ of $\R$, so we denote the ordering by the same symbol $<_{\R}$.

The action of $G_{m,n}$ on $\tE(T)$ is easy to describe, since $X$ and $Y$ act on $T_{m,n}$ as rotations of the tree.
\[ x: \left\{ 
\begin{array}{lll}
(N;+i\,\cdots) & \mapsto &(N;+(i+1)\,\cdots) \;\;\;\;\;(i\neq m-1)\\
(N;+(m-1)\,\cdots) & \mapsto & (\,(N+1)\,;-\cdots) \\
(N;-i\,\cdots) &\mapsto & (N;+1i\,\cdots) \\
\end{array}
\right. 
\]
\[ y: \left\{ 
\begin{array}{lll}
(N;+i\,\cdots) & \mapsto &(\,(N+1)\,;-1i\,\cdots) \\
(N;-i\,\cdots) &\mapsto & (N;-(i+1)\,\cdots) \;\;\;\;\;\;\;\;(i\neq n-1)\\
(N;-(n-1)\,\cdots) &\mapsto & (N;+\cdots)
\end{array}
\right. 
\]
Therefore, the action of $s_{1}$, $s_{2}$ and $s_{2}^{-1}$ are given by the formula
\[ s_{1}: \left\{ 
\begin{array}{lll}
(N;+i\,\cdots) & \mapsto &(N;+11(i+1)\,\cdots) \;\;\;\;\;(i\neq m-1)\\
(N;+(m-1)i\cdots) & \mapsto & (N;+1(i+1)\cdots) \;\;\;\;\;(i \neq n-1)\\
(N;+(m-1)(n-1)i\cdots) & \mapsto & (N;+(i+1)\cdots) \;\;\;\;\;(i \neq m-1)\\
(N;+(m-1)(n-1)(m-1)\cdots) & \mapsto & (\, (N+1)\, ;-\cdots)\\
(N;-i\,\cdots) &\mapsto & (N;+111\,\cdots) 
\end{array}
\right. 
\]
\[ s_{2}: \left\{ 
\begin{array}{lll}
(N;+\,\cdots) & \mapsto &(N;+(m-1)(n-1)\,\cdots) \\
(N;-i\,\cdots) &\mapsto & (N;+(m-1)(i-1)\,\cdots) \;\;\;\;\;\;\;\;(i\neq 1)\\
(N;-1i\,\cdots) &\mapsto & (N;+(i-1)\cdots) \;\;\;\;\;\;\;(i\neq 1)\\
(N;-11\,\cdots) &\mapsto & (N;-\cdots) 
\end{array}
\right. 
\]
\[ s_{2}^{-1}: \left\{ 
\begin{array}{lll}
(N;+i\,\cdots) & \mapsto &(N;-1(i+1)\,\cdots)  \;\;\;\;\;\;\;\;(i\neq m-1) \\
(N;+(m-1)i\,\cdots) &\mapsto & (N;-(i+1)\,\cdots) \;\;\;\;\;\;\;\;(i\neq n-1)\\
(N;+(m-1)(n-1)\,\cdots) &\mapsto & (N;+\cdots) \\
(N;-\cdots) &\mapsto & (N;-11\cdots) 
\end{array}
\right. 
\]

Let $E = (0;-1111\cdots)$ and $F=(0;+1111\cdots)$ be the point of $\tE(T)$ and let $<_{\{E,F\}}$ be a left-ordering of $G_{m,n}$ defined by the sequence $\{E,F\}$ and the standard action $\Theta$. The following theorem gives an alternative definition of $<_{D}$.

\begin{thm}
\label{thm:dynamics}
The left ordering $<_{\{E,F\}}$ is identical with the Dehornoy-like ordering $<_{D}$ defined by $\mS$.
\end{thm}

\begin{proof}
By the formula of the action of $G_{m,n}$ on $\tE(T_{m,n})$ given above, it is easy to see that for $1(\mS)$-positive element $g \in G_{m,n}$, $E <_{\R} g(E)$.
Thus by Property $C$ of $\mS$, $g(E)=E$ if and only if $g=s_{2}^{q}$ for $q \in \Z$. Similarly, $s_{2}^{q}(F) = (0;+(m-1)(n-1)\cdots) > F$ if $q>0$. 
Thus, we conclude $1<_{D} g$ then $1<_{\{E,F\}} g$, hence two orderings are identical.
\end{proof}

We remark that a more direct proof is possible. That is, we can prove that $<_{\{E,F\}}$ is a Dehornoy-like ordering without using Theorem \ref{thm:order}.

In fact, the proof of Theorem \ref{thm:dynamics} provides an alternative (but essentially equivalent since it uses the standard action of $G_{m,n}$) proof of the fact that $\mS$ has Property $A$.
On the other hand, using the description of the standard action given here, we can give an completely different proof of Property $C$, as we give an outline here.
 Let $g \in G$. If $g(E)=E$, then $g=s_{2}^{k}$ for $k \in \Z$. So assume $g(E) <_{\R} E$, so $g(E) = (N;l_{1}l_{2}l_{3}\cdots)$ where $N\leq 0$, $l_{1} \in \{+,-\}$. Let $c(g)$ be the minimal integer such that $l_{j'}=1$ for all $j'>j$.
We define the complexity of $g$ by $C(g)=(|N|,c(g))$. Now we can find $1(\mS)$-positive element $p_{g}$ such that $C(p_{g} g) < C(g)$. The construction of $p_{g}$ is not difficult but requires complex case-by-case studies, so we omit the details.  
Here we compare the complexity by the lexicographical ordering of $\Z \times \Z$. Since $C(g)=(0,0)$ implies $g(E)=E$, by induction of the complexity we prove $g$ is $\sigma(\mS)$-negative.

\begin{rem}
Regard $G_{3,2}=B_{3}$ as the mapping class group of $3$-punctured disc $D_{3}$ having a hyperbolic metric and let $\widetilde{D_{3}}\subset \mathbb{H}^{2}$ be the universal cover of $D_{3}$.
By considering the action on the set of points at infinity of $\widetilde{D_{3}}$, we obtain an action of $G_{3,2}$ on $\R$ which we call the {\it Nielsen-Thurston action}. 
The Dehornoy ordering $<_{D}$ of $B_{3}=G_{3,2}$ is defined by the Nielsen-Thurston action. See \cite{sw}.
In the case $(m,n)=(3,2)$, the standard action $\Theta$ derived from Bass-Serre tree is conjugate to the Nielsen-Thurston action. Thus, the Dehornoy-like ordering of $G_{m,n}$ is also regarded as a generalization of the Dehornoy ordering of $B_{3}$, from the dynamical point of view.
\end{rem}

We say a Dehornoy-like ordering $<_{D}$ defined by an ordered finite generating set $\mS$ has {\em Property S} (the {\em Subword property}) if $g <_{D} wg$ holds for all $\mS$-word positive element $w$ and for all $g \in G$. The Dehornoy ordering of the braid group $B_{n}$ has Property $S$ (\cite{ddrw}). 
  
One remarkable fact is that our Dehornoy-like ordering $<_{D}$ of $G_{m,n}$ does not have Property $S$ except the braid group case.

\begin{thm}
\label{thm:propertyS}
The Dehornoy-like ordering $<_{D}$ of $G_{m,n}$ does not have Property $S$ unless $(m,n)=(3,2)$.
\end{thm}
\begin{proof}
We use the dynamical description of $<_{D}$ given in Theorem \ref{thm:dynamics}.
If $m>2$ and $n\neq 2$, then 
\[ [s_{1}(s_{2}s_{1})] E = (0; +2 11\cdots ) <_{\R} (0;+(m-1)(n-1)1\cdots) = [s_{2}s_{1}]E \]
hence $s_{1}(s_{2}s_{1}) <_{D} (s_{2}s_{1})$.
\end{proof}

However we show that the Dehornoy-like ordering $<_{D}$ of $G_{m,n}$ has a slightly weaker property which can be regarded as a partial subword property.

\begin{thm}
\label{thm:weakS}
Let $<_{D}$ be the Dehornoy-like ordering of $G_{m,n}$.
Then $g <_{D} s_{2}g$ holds for all $g \in G$. 
\end{thm}
\begin{proof}
Observe that the standard action of $s_{2}$ on $\tE(T)$ is monotone increasing. That is, for any $p \in \tE(T)$, we have $p  \leq _{\R} s_{2}(p)$. Thus, $g(E)  \leq s_{2}g(E)$. The equality holds only if $g(E) = E$, so in this case $g=s_{2}^{k}$ $(k \in \Z)$. So in this case we also have a strict inequality $g <_{D} s_{2}g$.
\end{proof}

\begin{rem}
A direct proof of Theorem \ref{thm:propertyS} which does not use the dynamics is easy once we found a counter example.
If $(m,n)\neq (3,2)$, then $s_{2}s_{1}s_{2}= s_{1}s_{2}s_{1}W$ for an $\mS$-positive word $W$, so
\begin{eqnarray*}
 s_{1}^{-1}s_{2}^{-1}s_{1}^{-1}s_{2}s_{1}
& = & s_{1}^{-1}s_{2}^{-1}s_{1}^{-1}(s_{2}s_{1}s_{2})s_{2}^{-1} \\
& = & s_{1}^{-1}s_{2}^{-1}s_{1}^{-1}(s_{1}s_{2}s_{1}W)s_{2}^{-1} \\
& = & W s_{2}^{-1}
\end{eqnarray*}
The last word is $1(\mS)$-positive hence $s_{1}(s_{2}s_{1}) <_{D} (s_{2}s_{1})$.

Using dynamics we can easily find a lot of other counter examples. The main point is that, as we can easily see, the action $s_{1}$ is not monotone increasing unlike the action of $s_{2}$. That is, there are many points $p \in \tE(T)$ such that $s_{1}(p) <_{\R} p$.
\end{rem}

\begin{rem}
Another good property of the standard generator $\mS= \{\sigma_{1},\sigma_{2}\}$ of $B_{3}$ is that the $\mS$ word-positive monoid $P_{\mS}=B_{3}^{+}$ is a {\it Garside monoid}, so $B_{3}^{+}$ has various nice lattice-theoretical properties. See \cite{bgm},\cite{d2} for a definition and basic properties of the Garside monoid and Garside groups.
In particular, the monoid $B_{3}^{+}$ is atomic.
That is, if we define the partial ordering $\prec$ on $B_{3}^{+}$ by $g \prec g'$ if $g^{-1}g' \in B_{3}^{+}$, then for every $g \in G$, the length of a strict chain 
\[ 1 \prec g_{1} \prec \cdots \prec g_{k }= g\]
is finite.  
However, if $m>3$, then the $\mS$-word positive monoid $P_{\mS}$ is not atomic.
If $m>3$, then $s_{2}s_{1}s_{2}= s_{1}s_{2}s_{1}s_{2} W$ holds for $W \in P_{\mS}$ so we have a chain 
\[  \cdots \prec s_{1}^{2}s_{2}s_{1}s_{2}  \prec s_{1} s_{1}s_{2}s_{1}s_{2} W  = s_{1}s_{2}s_{1}s_{2} \prec s_{1}s_{2}s_{1}s_{2} W = s_{2}s_{1}s_{2} \]
having infinite length.
Thus for $m>3$, $P_{\mS}$ is not a Garside monoid.

On the other hand, the groups $G_{m,n}$ have a lot of Garside group structures. For example, let $\mX=\{x,y\}$ be a generating set of $G_{m,n}$. Then the $\mX$-word positive monoid $P_{\mX}$ is a Garside monoid. Moreover, if $m$ and $n$ are coprime, that is, if $G_{m,n}$ is a torus knot group, then there are other Garside group structures due to Picantin \cite{p}.
Thus, unlike the Dehornoy ordering of $B_{n}$, a relationship between general Dehornoy-like orderings and Garside structures of groups seem to be very weak. Thus, it is interesting problem to find other family of left-orderings which is more related to Garside group structure.

In the remaining case $(m,n)=(3,3)$, the author could not determine whether $P_{\mS}$ is a Garside monoid or not. 
\end{rem}

\subsection{Exotic orderings: left orderings with no non-trivial proper convex subgroups}

In \cite{c}, Clay constructed left orderings of free groups which has no non-trivial proper convex subgroups by using the Dehornoy ordering of $B_{3}$. Such an ordering is interesting, because many known constructions of left orderings, such as a method to use group extensions, produce an ordering having proper non-trivial convex subgroups. 

In this section we construct such orderings by using a Dehornoy-like ordering of $G_{m,n}$. By using the dynamics, we prove a stronger result even for the 3-strand braid group case.  
Let $H$ be a normal subgroup of $G_{m,n}$. By abuse of notation, we also use $<_{D}$ to represent both the Dehornoy like ordering of $G_{m,n}$ and its restriction to $H$. 

First we observe the following lemma, where the partial subword property plays an important role.
\begin{lem}
\label{lem:conj}
Let $C$ be a non-trivial $<_{D}$-convex subgroup of $H$. 
If $G^{(2)}\cap H =\{1\}$, then $s_{2}^{k}c s_{2}^{-k} \in C$ for all $k \in \Z$ and $c \in C$.
\end{lem}
\begin{proof}
Let $c \in H$ be a $<_{D}$-positive element. Since $G^{(2)}\cap H =\{1\}$, $c$ must be $1(\mS)$-positive.
So $s_{2}^{k} c^{-1} s_{2}^{-k}$ is $1(\mS)$-negative, hence $c s_{2}^{k} c^{-1} s_{2}^{-k} <_{D} c$.
On the other hand, by Theorem \ref{thm:weakS}, $c s_{2}^{k} c^{-1} >_{D} 1$. Thus $c s_{2}^{k} c^{-1}(E) = c s_{2}^{k} c^{-1}s_{2}^{-k}(E) \geq_{\R} E$, so $ cs_{2}^{k}c^{-1}s_{2}^{-k} \geq_{D} 1$.  
$H$ is a normal subgroup, so $c s_{2}^{k}c^{-1}s_{2}^{-k} \in H$. 
Since $C$ is $<_{D}$-convex subgroup, $cs_{2}^{k} c^{-1} s_{2}^{-k} \in C$. Hence we conclude $s_{2}^{k}cs_{2}^{-k} \in C$. 
\end{proof}

Now we show that in most cases, the restriction of the Dehornoy-like ordering to a normal subgroup of $G_{m,n}$ yields a left-ordering having no non-trivial proper convex subgroups.

\begin{thm}
\label{thm:exotic}
Let $H$ be a normal subgroup of $G_{m,n}$ such that $G^{(2)}\cap H =\{1\}$.
Then the restriction of the Dehornoy-like ordering $<_{D} $ to $H$ contains no non-trivial proper convex subgroup.
\end{thm}
\begin{proof}
Let $C$ be a non-trivial $<_{D}$-convex subgroup of $H$ and $c \in C$ be $<_{D}$-positive element.
Since $c$ must be $1$-positive, by Lemma $\ref{lem:conj}$, we may assume that
$c = s_{2}s_{1}s_{2}w$ where $w$ is a $1$-positive element, by taking a power of $c$ and conjugate by $s_{2}$ if necessary.
Similarly, we also obtain $c' \in C$ such that $c'=w' s_{2}s_{1}s_{2} $ where $w'$ is a $1$-positive element.

By computing the standard action of $c'c$, then we obtain
\begin{eqnarray*}
c'c(E) = w' s_{2}s_{1}s_{2}^{2}s_{1}s_{2}w (E) & >_{\R } & w' s_{2}s_{1}s_{2}^{2}s_{1}s_{2} (E)  =  w' (1; -11(n-1)11\cdots ) \\
& >_{\R} & (1; -1111\cdots).
\end{eqnarray*}
Thus, for any $h \in H$,  $(c'c)^{N}(E) >_{\R} (N; -1111\cdots) >_{\R} h(E) >_{\R} E$ holds for sufficiently large $N>0$. Since $C$ is convex and $c'c \in C$, this implies $h \in C$ so we conclude $C=H$.

\end{proof}

We remark that the assumption that $G^{(2)}\cap H =\{1\}$ is necessary, since $H \cap G^{(2)} $ yields a $<_{D}$-convex subgroup of $H$. We also remark that the hypothesis $G^{(2)}\cap H =\{1\}$ implies that $H$ is a free group, since $G_{m,n}=\Z *_{\Z} \Z$.

Theorem \ref{thm:exotic} provides an example of a left-ordering of the free group of {\em infinite} rank which does not have any non-trivial proper convex subgroups. For example, take $F=[B'_{3},B'_{3}]$, where $B'_{3}=[B_{3},B_{3}] \cong F_{2}$ be the commutator subgroup of $B_{3}$, which is isomorphic to the rank $2$ free group.

\end{document}